\documentclass[dvips, 12pt, a4paper]{paper}
\usepackage{amsmath,amssymb,amsthm, dsfont}
\usepackage{color}
\definecolor{violet}{rgb}{0.5,0,0.8}

\theoremstyle{plain}
\newtheorem{theorem}{Theorem}[section]

\theoremstyle{plain}
\newtheorem{cor}[theorem]{Corollary}

\theoremstyle{plain}
\newtheorem{lemma}[theorem]{Lemma}

\theoremstyle{plain}

\theoremstyle{plain}
\newtheorem{proposition}[theorem]{Proposition}

\theoremstyle{plain}
\newtheorem{definition}[theorem]{Definition}

\theoremstyle{plain}

\theoremstyle{plain}
\newtheorem{remark}[theorem]{Remark}

\begin{document}

\title{Enlargement of filtration and predictable representation property for semi-martingales}
\author{Antonella Calzolari \thanks{Dipartimento di Matematica - Universit\`a di Roma
"Tor Vergata", via della Ricerca Scientifica 1, I 00133 Roma,
Italy }  \and Barbara Torti $^*$} \maketitle
\bigskip
\maketitle
\begin{abstract}
We present two examples of loss of the predictable representation
property for semi-martingales by enlargement of the reference
filtration.~First of all we show that the predictable
representation property for a semi-martingale $X$ does not
transfer from the reference filtration $\mathbb{F}$ to a larger
filtration $\mathbb{G}$ if the information starts growing up to a
positive time.~Then we study the case
$\mathbb{G}=\mathbb{F}\vee\mathbb{H}$ when there exists a second
special semi-martingale $Y$ enjoying the predictable
representation property with respect to $\mathbb{H}$.~We establish
conditions under which the triplet $(X,Y,[X,Y])$ enjoys the
predictable representation property with respect to $\mathbb{G}$.
\end{abstract}

\textbf{AMS 2010} 60G48, 60G44, 60H05, 60H30, 91G99

\section{Introduction}
Given a filtered probability space $(\Omega, \mathcal{F},
\mathbb{F}=(\mathcal{F}_t)_{t\in[0,T]}, P)$ and a fixed
 semi-martingale $X$ on it, a classical problem in
stochastic analysis is the investigation of conditions which allow
to represent every $L^{\infty}(\Omega, \mathcal{F}_T, P)$-random
variable as the sum of an $\mathcal{F}_0$-measurable random
variable and a stochastic integral with respect to $X$.~When there
exists an equivalent local-martingale measure $Q$ for $X$ that
representation follows as soon as $X$ enjoys the
\textit{predictable representation property} (in short, p.r.p.)
with respect to the filtration $\mathbb{F}$ under $Q$
\footnote{when not necessary we will avoid to mention $Q$}, that
is any $(Q,\mathbb{F})$-local martingale can be written in the
form $m+\int_0^t\xi_sdX_s$ where $m$ is $\mathcal{F}_0$-measurable
and $\xi$ is $\mathbb{F}$-predictable (see Definition 13.1 and
Theorem 13.4 in \cite{he-wang-yan92}).~The equivalence between the
representation of all $(Q,\mathbb{F})$-local martingales and the
uniqueness
 of $Q$, modulo $\mathcal{F}_0$, is the content of a classical martingale representation result (see Theorem 13.9 in \cite{he-wang-yan92}).~In particular, under the stronger hypothesis that $Q$ is the unique equivalent
 martingale measure for $X$, $\mathcal{F}_0$ is trivial and $X$ enjoys the
 $(Q,\mathbb{F})$-p.r.p.~ (see \cite{ha-pli2} and Theorem 13.4 in \cite{he-wang-yan92}).\\
If the reference filtration $\mathbb{F}$ coincides with the
natural filtration of
 $X$, then p.r.p.~holds either for Brownian motion and Poisson process or, under
suitable assumptions,
 for diffusions with jumps and in some particular non Markovian contexts
 (see e.g.~\cite{jacod77} and \cite{davis2} and references
 therein).\\
\noindent Assuming that the p.r.p.~for
 $X$ holds, an interesting question is about
 maintenance of that property with respect to an enlarged filtration $\mathbb{G}$ in the following sense.~If $\tilde{Q}$ is an equivalent
 local-martingale measure for $X$ with respect to $\mathbb{G}$, does $X$ enjoy the $(\tilde{Q},\mathbb{G})$-p.r.p.?~When the answer is positive, we say that \textit{p.r.p. transfers from $\mathbb{F}$ to $\mathbb{G}$}.~In
 this case $\tilde{Q}$ is the unique equivalent local-martingale measure, modulo ${\mathcal{G}}_0$, for $X$ with respect to
 $\mathbb{G}$.~Otherwise p.r.p.~cannot hold with respect to any
  equivalent local-martingale measure for $X$ with respect to
 $\mathbb{G}$, and we say that \textit{p.r.p. disappears}.

\noindent This issue is related to the \textit{problem of
enlargement of filtration}, that is the investigation of
conditions under which, on a given probability space,
  $(P,\mathbb{F})$-semi-martingales are also $(P,\mathbb{G})$-semi-martingales (see the classical \cite{Jeulin} and more recently \cite{li-ru11}).\\
  Maintenance or loss of the p.r.p.~with respect to an enlarged filtration appear both in the
 literature.~Let us now recall two well-known examples of enlargement of filtration.~In
 \cite{ame}
 it is shown that the p.r.p.~is preserved in case of  \textit{initial enlargement} of $\;\mathbb{F}$ defined by $\mathcal{G}_t=\mathcal{F}_t\vee\sigma(G)$, with $G$
a random variable satisfying the condition
$P(G\in\cdot|\mathcal{F}_t)(\omega)\sim
P(G\in\cdot),\,\textrm{for\, almost\, all}\;\omega$.~This
assumption enables to show  that $X$ enjoys the
$(Q^{\mathbb{G}},\mathbb{G})$-p.r.p.~under a suitable equivalent
martingale measure $Q^{\mathbb{G}}$ on $\mathcal{G}_T$.~At the
same time, a Brownian motion $B$ fails to exhibit  the p.r.p.~in
case of  \textit{progressive enlargement} of its natural
filtration $\mathbb{F}^B$ obtained by the observation of the
occurrence of a positive random time $\tau$, which is not an
$\mathbb{F}^B$-stopping time, that is when
$\mathcal{G}_t=\cap_{s>t}\mathcal{F}^B_s\vee\sigma(\tau\wedge
s)$.~In particular this occurs when $\tau$ is a continuous random
variable independent of $B$, and consequently integrals with
respect to $B$ are not enough to represent  all the
$\mathbb{G}$-martingales.~Indeed one has to add integrals with
respect to the compensated default process $\mathbb{I}_{\tau \le
\cdot}-\int_0^{\tau \wedge \cdot}\lambda_sds$, with $\lambda=f/G$,
where $f$ and $G$ are the density function and the survival
function  of $\tau$, respectively (this result can be viewed as a
simple application of Theorem
7.5.5.1 in \cite{Jean-yor-chesney}).\\
We stress that the two examples above  differ;~more  precisely the
initial sigma-algebra in the latter remains trivial, whereas in
the former the enlargement already takes
  place up to the initial time: that is the starting sigma-algebra of the new filtration is different from $\mathcal{F}_0$.\\
In general, the fact that the progressive enlargement  by a random
time destroys the p.r.p.~is well-known; we refer to
\cite{jean-song12} for conditions under  which the p.r.p. can be
achieved  by considering a larger number of driving processes.\\
 \noindent In this paper
$X$ is a special
 $(P,\mathbb{F})$-semi-martingale which admits a
 unique equivalent martingale measure.~Therefore $\mathcal{F}_0$ is trivial and $X$ enjoys the p.r.p.~with respect to $\mathbb{F}$.~Without selecting particular models, we present two results where the p.r.p.~disappears
 when the reference filtration is enlarged by
keeping trivial the sigma-algebra at time zero.~Our aim is to give
a contribution to the investigation of the link between the
p.r.p.~and the reference filtration of the semi-martingale.~It is
well-known that p.r.p.~may fail with respect to the natural
filtration (see Example 23.11 in \cite{stoy}).~At the same time in
some cases the property holds with respect to a filtration larger
than the
natural one (see Remark \ref{rem:riflesso} in Section \ref{sec:inf=min}).\\
\noindent In our first result no assumption is introduced  on the
source of randomness giving rise to the enlargement.~We prove
that, if there exists an equivalent martingale measure for  $X$
with respect to the enlarged filtration $\mathbb{G}$, then the
p.r.p.~disappears whenever
the set $\{t\in[0,T]:\mathcal{F}_t \varsubsetneq\mathcal{G}_t\}$ has a positive minimum.\\
\noindent In our second result the enlargement is obtained by
adding the information given by a new filtration $\mathbb{H}$ such
that $\mathcal{H}_0$ is trivial.~More precisely $\mathbb{G}$
coincides with $\mathbb{F}\vee\mathbb{H}$.~We assume that there
exists a $(P,\mathbb{H})$-semi-martingale $Y$ enjoying the
p.r.p.~with respect to $\mathbb{H}$.~Moreover the martingale part
$N$ of $Y$ is $(P,\mathbb{G})$-strongly orthogonal to the
martingale part $M$ of $X$, and $X$ and $Y$ satisfy a technical
assumption which provides them the \textit{structure
condition}.~We also introduce suitable assumptions  assuring the
existence of the \textit{minimal martingale measures} for $X$ and
$Y$.~In this setting we prove that the
$(P,\mathbb{G})$-semi-martingale $(X,Y,[X,Y])$ admits a unique
equivalent martingale measure $Q$.~More precisely we show that the
triplet $(X,Y,[X,Y])$ is a \textit{basic set of
$(Q,\mathbb{G})$-orthogonal martingales} so that three is the
\textit{multiplicity of $\mathbb{G}$} in the sense of Davis and
Varaiya (see \cite{davis1}).~Therefore  the p.r.p.~by
the enlargement of filtration does not hold for $X$.\\
\noindent We stress  that the issue addressed in this paper
naturally emerges in the analysis of financial markets, where
 the discounted price of the risky asset is modeled as a semi-martingale.~In this setting generally the payoffs
 are
random variables measurable with respect to the final
sigma-algebra of the reference filtration and the investment
strategy is a predictable process.~Then, when the market is free
of arbitrage, the p.r.p.~of the discounted asset price process
provides the \textit{completeness of the market} that is the
perfect replication
 of all the  essentially bounded payoffs (see \cite {gro_po99}).~In their seminal papers Harrison and Pliska clearly state that
completeness is a joint property of the filtration and of the
asset price process and in particular they argued that the
structure of the filtration should influence completeness.~They
also provided the original version of the \textit{II Theorem of
Asset Pricing} (see \cite{ha-pli1}, \cite{ha-pli2}), even if their
statement is not completely correct in the definition of
self-financial strategies (to clarify this fact see
\cite{muller89} and the Appendix in \cite {jarr-mad91} and
\cite{cha-stri94} about the distinction between \textit{vector
completeness} and \textit{component completeness}, which is at the
origin of the imprecision of the Harrison and Pliska's
result).~Completeness has been widely studied
 when the reference filtration coincides with the natural one.~However markets with default  or
markets with better informed agents  are modeled by considering on
the probability space a filtration larger than the natural one and
most of them are not complete markets or even not arbitrage free
(see e.g.~\cite{bia-cre-07} or
\cite{jean-song12} and \cite{fontana12}).\\
 This
paper is organized as follows.~In Section~\ref{sec:sett-not} we
describe the mathematical setting and recall a classical result
about the enlargement of filtration.~\mbox{Section
\ref{sec:inf=min}} and \mbox{Section~\ref{sec:adding-semimg}} are
devoted to our results.~\mbox{Section~\ref{sec:adding-semimg}} is
divided into two subsections: in the former  we discuss the
simpler case of martingales and in the latter we extend the result
to semi-martingales.~In \mbox{Section~\ref{sec:conclusions}} we
discuss  possible connections of our results with some papers in
the recent literature. Finally we devote Section
\ref{sec:perspectives} to avenues for future research in this
area.
\section{Setting and notations}\label{sec:sett-not}
Let  $T>0$ be a finite time horizon and let $X=(X_t)_{t\in[0,T]}$
be a real valued  c\`adl\`ag square-integrable semi-martingale
defined on a given probability space $(\Omega,
\mathcal{F},\mathbb{F},P)$, with the filtration
$\mathbb{F}$ satisfying the usual conditions of right-continuity and completeness.~\\
More precisely let $X$ belong to the space $\mathcal{S}^2(P,
\mathbb{F})$ \textit{of semi-martingales} (see, e.g.,
\cite{del-me-b}),
 i.e.~let $X$ be a
special semi-martingale
 with canonical decomposition
 \begin{equation}\label{semi-martingale}
 X=X_0+M+A
 \end{equation}
such that
 \begin{equation}\label{eq-int-cond-X}
 E^P\left[X_0^2+[ M]_T+|A|^2_T\right]<+\infty.
 \end{equation}
As usual $M$ is a $(P,\mathbb{F})$-martingale, $A$ is an
$\mathbb{F}$-predictable process of finite variation, $M_0=A_0=0$,
$|A|$ denotes the total variation process of $A$ and $[M]$  the
quadratic variation process of $M$.~Note that by integrability
condition (\ref{eq-int-cond-X}) it follows
\begin{equation}\label{eq-cond_struttura}
E^P\Big[\sup_{t\in[0,T]}X_t^2\Big]<\infty.
\end{equation}
\noindent In all the paper, given a filtered probability space
$(\Omega, \mathcal{A}, \mathbb{A}, R)$ and a square-integrable
$(R,\mathbb{A})$-semi-martingale $S=(S_t)_{t\in [0,T]}$ on it,
$L_0^2(\Omega,
 \mathcal{A}_T, R)$ will denote the subset of the centered elements of $L^2(\Omega,
 \mathcal{A}_T, R)$ and  $\mathcal{L}^2(S, R,\mathbb{A})$ the space of the
$\mathbb{A}$-predictable processes $\xi$ such that
\begin{equation}\label{def:integrands}E^R\left[\int_0^T\xi^2_t\,d[S]_t\right]<+\infty.\end{equation}
 \noindent Denote by
$\mathbb{P}(X,\mathbb{F})$ the set of probability measures on
$(\Omega, \mathcal{F}_T)$ under which $X$ is a martingale and
which are equivalent to $P|_{\mathcal{F}_T}$,
 the restriction of $P$ to $\mathcal{F}_T$.~Assume that the set $\mathbb{P}(X,\mathbb{F})$ is a singleton, more precisely\vspace{0.5em}\\
\textbf{H1)}\;\textit{$\mathbb{P}(X,\mathbb{F})=\{P^{X}\}.$}
\vspace{0.5em}\\
Assumption \textbf{H1)} has two important consequences as already
recalled in the introduction.~The initial $\sigma$-algebra
$\mathcal{F}_0$ turns out to be $P$-trivial and the
$(P,\mathbb{F}$)-semi-martingale $X$ enjoys the p.r.p.~with
respect to $\mathbb{F}$ under $P^X$, that is any
$(P^X,\mathbb{F})$-local martingale admits a representation of the
form $m+\int_0^t\xi_sdX_s$ where $m$ is a constant and $\xi$ is
$\mathbb{F}$-predictable.~In particular each $H$ in $L^2(\Omega,
 \mathcal{F}_T, P^X)$ admits $P^X$-a.s.~the representation
\begin{align}\label{eq-rappr-mg1-l}
H=H_0+\int_0^T\xi^H_sdX_s,
\end{align}
with $H_0$ a constant and $\xi^H$ a process in
$\mathcal{L}^2(X,P^X,\mathbb{F})$.\vspace{0.5em}\\
Let $\mathcal{M}^2(P^X,\mathbb{F})$ be the space of the
square-integrable $(P^X,\mathbb{F})$-martingales.~Then for any $Z$
in $\mathcal{M}^2(P^X,\mathbb{F})$ an immediate application of
(\ref{eq-rappr-mg1-l}) to $H=Z_T$ proves that
 $P^X$-a.s.
 \begin{equation}\label{natural-completeness-3}
 Z_t=Z_0+\int_0^t\xi^Z_sdX_s\end{equation}
 with $Z_0$ a constant and $\xi^Z$
 a process in the space $\mathcal{L}^2(X,P^X,\mathbb{F})$.
\vspace{0.5em}\\
Set
\begin{equation}\label{def:K2}K^2(\Omega,\mathbb{F},P^X,X):=\left\{\int_0^T\xi_sdX_s,\;\xi\in
\mathcal{L}^2(X,P^X,\mathbb{F})\right\}.\end{equation} Since
$\mathcal{F}_{0}$ is trivial, representation
(\ref{eq-rappr-mg1-l}) is equivalent to the equality
\begin{equation}\label{natural-completeness-2}
L_0^2(\Omega,
 \mathcal{F}_T, P^X)=K^2(\Omega,\mathbb{F},P^X,X).\end{equation}
\begin{definition}\label{def-enlargement}
 A filtration $\mathbb{G}=(\mathcal{G}_t)_{t\in[0,T]}$ on
$(\Omega,\mathcal{F},P)$ under the standard hypotheses is an enlargement of the filtration $\mathbb{F}$ if
$$
\mathcal{F}_{t}\subset \mathcal{G}_t \text{ for
all }t\in [0,T]\;\;\;\text{ and }
\;\;\;\mathcal{F}_t\varsubsetneq\mathcal{G}_t \text{
for some } t\in [0,T].
$$
\end{definition}
\noindent Let $Q$ be a probability measure on the space $(\Omega,
\mathcal{F})$ and
 $\mathcal{M}(Q,\mathbb{F})$ the space of uniformly
integrable $(Q,\mathbb{F})$-martingales.~The following theorem holds.
\begin{theorem}\label{thm-jacod1}(Theorem 3 in \cite{Bre-yor78}) Let  $\mathbb{G}$ be  any
enlargement of $\mathbb{F}$.~Then the following
conditions are equivalent
\begin{enumerate}
    \item[(i)] $\mathcal{M}(Q,\mathbb{F})\subset
\mathcal{M}(Q, \mathbb{G})$\;\;(immersion property);
    \item[(ii)]
    for any $t$ in $[0,T]$ and any bounded $\mathcal{F}_T$-measurable random variable $Y$,
    $E^Q[Y\mid\mathcal{G}_t]$ is $\mathcal{F}_t$-measurable.
\end{enumerate}
Under any of these conditions,
$\mathcal{F}_t=\mathcal{F}_T\cap\mathcal{G}_t$.
\end{theorem}
\section{Loss of the predictable representation property:
a sufficient condition on the enlarged
filtration}\label{sec:inf=min}
Let $\mathbb{G}$ be any enlargement of $\mathbb{F}$ and let $\mathbb{P}(X,\mathbb{G})$ be the set of
probability measures on $(\Omega, \mathcal{G}_T)$ under which $X$
is a $\mathbb{G}$-martingale and which are equivalent to
$P|_{\mathcal{G}_T}$.~Consider the
assumption\vspace{0.5em}\\
\textbf{H2)}\;\textit{$\label{G-non-arbitrage}\mathbb{P}(X,\mathbb{G})\neq
\emptyset.$}\vspace{0.5em}\\
 In this section for any
$Q$ in $\mathbb{P}(X,\mathbb{G})$ the space
$K^2(\Omega,\mathbb{G},Q,X)$ is defined analogously to (\ref{def:K2}).\vspace{0.5em}\\
A consequence of Theorem \ref{thm-jacod1} is the following
proposition.
\begin{proposition}\label{strict-incl} Assume \textbf{H1)} and
\textbf{H2)} and set
\begin{equation}\label{eq-inf-incl-filtr}
u:=\inf\{t\in[0,T]:
\mathcal{F}_t\varsubsetneq\mathcal{G}_t\}.
\end{equation}
Then when $u=T$ it holds $
\mathcal{F}_T\varsubsetneq\mathcal{G}_T$ and when $u<T$ it holds
\begin{equation}\label{eq-inclusion}
\mathcal{F}_t\varsubsetneq\mathcal{G}_t \text{ for all } t\in
(u,T].
\end{equation}
\end{proposition}
\begin{proof} Note that (\ref{eq-inf-incl-filtr}) is well-posed by Definition \ref{def-enlargement}.~Let $Q$ be an element of
$\mathbb{P}(X,\mathbb{G})$ (existence follows by Hypothesis
$\textbf{H2)}$).~Since $\mathcal{F}_T\subset\mathcal{G}_T$, then
$Q|_{\mathcal{F}_T}$ belongs to
$\mathbb{P}(X,\mathbb{F})$.~Hypothesis \textbf{H1)} gives
$Q|_{\mathcal{F}_T}=P^X$ so that
$\mathcal{M}^2(Q,\mathbb{F})=\mathcal{M}^2(P^X,\mathbb{F})$.~The
last equality together with the representation property
(\ref{natural-completeness-3}) and the definition of $Q$ implies
$\mathcal{M}^2(Q,\mathbb{F})\subset \mathcal{M}^2(Q,\mathbb{G})$
and by a density argument the immersion property
follows\footnote{see Proposition 3.1 in \cite{jean_lecam09} for a
different proof}.~Then, fixed $T^\prime\leq T$, applying Theorem
\ref{thm-jacod1} equality
$\mathcal{F}_t=\mathcal{F}_{T^\prime}\cap\mathcal{G}_t$ easily
follows, for all $t\in [0,T^\prime]$.~The latter identity proves
relation (\ref{eq-inclusion}) in the case $u<T$.~In fact
$\mathcal{F}_{T^\prime}= \mathcal{G}_{T^\prime}$ for some
$T^\prime\in (u,T]$ would imply $\mathcal{F}_t=\mathcal{G}_t$ for
all $t\in [0,T^\prime]$, in contradiction to the definition of
$u$.~Finally when $u=T$ the inclusion in (\ref{eq-inclusion}) is
trivial.
\end{proof}
\begin{remark}
When $ u=\min\{t\in[0,T]:
\mathcal{F}_t\varsubsetneq\mathcal{G}_t\}$ obviously relation
(\ref{eq-inclusion}) gets stronger, that is it holds
$\mathcal{F}_t\varsubsetneq\mathcal{G}_t \text{ for all } t\in
[u,T]$.
\end{remark}
\begin{remark} The above result implies that if \textbf{H1)} holds
 then condition $\mathcal{F}_{T}= \mathcal{G}_{T}$ forces
$\mathbb{P}(X,\mathbb{G})$ to be void.~In financial context this
result has a very intuitive appeal.~When the existence of a
martingale measure is equivalent to the absence of arbitrage (see
e.g.~\cite{de-scha94}), if the market is complete, adding information without changing the
payoffs' set generates arbitrage opportunities.
\end{remark}
\begin{theorem}\label{thm-nostro}
 Assume that $\mathcal{G}_0$ is  $P$-trivial and that
$u$ defined by
  (\ref{eq-inf-incl-filtr}) is a minimum.~If \textbf{H2)} holds, then for any
  $Q$ in $\mathbb{P}(X,\mathbb{G})$
  \begin{equation}\label{eq:notprp} K^2(\Omega,\mathbb{G},Q,X)\varsubsetneq
L^2_0(\Omega,\mathcal{G}_T, Q).\end{equation}
\end{theorem}
\begin{proof}
Since by assumption $\mathcal{G}_0$ is $P$-trivial, the hypothesis
that $u$ is a minimum implies $u>0$.~Let $Q$ be an element of
$\mathbb{P}(X,\mathbb{G})$ and let $A$ be a non trivial set such
that $A\not\in\mathcal{F}_u$ and $A\in\mathcal{G}_u$.~Consider the
random variable
$$L:=\mathbb{I}_A-E^{Q}[\mathbb{I}_A\mid\mathcal{F}_u].$$$\mathbb{I}_A$ is not
$\mathcal{F}_u$-measurable, so that $L$ is a non trivial element
of
 $L^2_0(\Omega,
 \mathcal{G}_T,Q)$.~If $L$ were in
 $K^2(\Omega,\mathbb{G},Q,X)$ then
 it would exist a $\mathbb{G}$-predictable process $\eta^L$ in the set $\mathcal{L}^2(X,Q,\mathbb{G})$ such that
\begin{equation}\label{L-representation2}L=\int_0^T\eta^L_sdX_s,\;\;\;\;Q\textrm{-a.s.}.\end{equation}
But we claim that $L$ satisfies
  \begin{equation}\label{L-ortogonality2}E^Q\left[L\int_0^T\eta_sdX_s\right]=0\end{equation}
  for every $\mathbb{G}$-predictable $\eta$ in $\mathcal{L}^2(X,Q,\mathbb{G})$.~Together  with
representation (\ref{L-representation2}), this would imply
$E^{Q}\left[L^2\right]=0$, i.e. $L=0$,
$Q$-a.s..\vspace{0.5em}\\
In order to prove (\ref{L-ortogonality2}) it is useful to rewrite
$E^Q\left[L\,\int_0^T\eta_sdX_s\right]$ as
$$E^Q\left[E^Q\left[L\,\left(\int_0^T\mathbb{I}_{s< u}\eta_sdX_s+\eta_u\Delta X_u\right)\mid
\mathcal{F}_u\right]\right]
+E^Q\left[E^Q\left[L\,\int_0^T\mathbb{I}_{s>
u}\eta_sdX_s\mid\mathcal{G}_u\right]\right].$$ Then equality
(\ref{L-ortogonality2}) immediately follows taking into account
that $\int_0^T\mathbb{I}_{s< u}\eta_sdX_s$ and $\eta_u\Delta X_u$
are $\mathcal{F}_u$-measurable and $L$ is
$\mathcal{G}_{u}$-measurable so that
$E^Q\left[L\,\int_0^T\eta_sdX_s\right]$ coincides with

\begin{align*}
E^{Q}\left[\left(\int_0^{u^-}\eta_sdX_s+\eta_u\Delta
X_u\right)\,E^{Q}\left[L\mid\mathcal{F}_u\right]\right]
+E^{Q}\left[L\,E^{Q}\left[\int_{u^+}^T\eta_sdX_s\mid\mathcal{G}_u\right]\right]
    \end{align*}
    and finally considering that $E^{Q}\left[L\mid\mathcal{F}_u\right]=0=E^{Q}\left[\int_{u^+}^T\eta_sdX_s\mid\mathcal{G}_u\right]$.
\end{proof}
\begin{cor}
Assume \textbf{H1)}.~Then, under the hypotheses of the previous
theorem, the p.r.p.~for $X$ is not preserved by the enlargement of
filtration from $\mathbb{F}$ to $\mathbb{G}$.
\end{cor}
\begin{proof}
$X$ enjoys p.r.p.~with respect to $\mathbb{F}$ under
$Q|_{\mathcal{F}_T}=P^X$ but the strict inclusion
(\ref{eq:notprp})
  implies that $X$ doesn't enjoy the (Q, $\mathbb{G}$)-p.r.p..
\end{proof}
\begin{remark} Theorem \ref{thm-nostro} cannot apply when $\mathbb{G}$ is a quasi-left continuous filtration.~In fact,
in this case, since
$\mathcal{G}_u=\mathcal{G}_{u^{-}}=\mathcal{F}_{u^{-}}$, the
assumption that $u$ is a minimum would imply the existence of a
non trivial set in
$\mathcal{F}_{u^{-}}$ but not in $\mathcal{F}_u$.\\
As an immediate application of Theorem \ref{thm-nostro} one can
take the progressive enlargement defined by
$\mathcal{G}_t=\cap_{s>t}\mathcal{F}_s\vee\sigma(\tau\wedge s)$
with $\tau$ any positive random variable  independent of
$\mathcal{F}_T$ taking values in a finite set.~
\end{remark}
\begin{remark}\label{rem:riflesso}
Assumption that $u$ is a minimum in Theorem \ref{thm-nostro}
cannot be dropped as proved by the following example.

    \noindent Let $B$ be a Brownian motion, then the process  $\int_0^\cdot \textrm{sgn}(B_s) dB_s$ is still a Brownian motion enjoying both the
$\mathbb{F}^{|B|}$-p.r.p.~and the $\mathbb{F}^{B}$-p.r.p.  (see
Chapter 6 in \cite{mans-yor}).~Moreover
\begin{equation*}
\inf\{t\in[0,T]:
\mathcal{F}^{|B|}_t\varsubsetneq\mathcal{F}^{B}_t\}=0,
\end{equation*}
\noindent and the set $\{t\in[0,T]:
\mathcal{F}^{|B|}_t\varsubsetneq\mathcal{F}^{B}_t\}$ doesn't have
a minimum. However  the p.r.p.~of $\int_0^\cdot \textrm{sgn}(B_s)
dB_s$ transfers from $\mathbb{F}^{|B|}$ to $\mathbb{F}^{B}$.
\end{remark}
\section{Adding the reference filtration of a
semi-martingale}\label{sec:adding-semimg}
In this section $\mathbb{H}$ is a filtration on $(\Omega,\mathcal{F},P)$ satisfying the usual conditions of right-continuity and completeness with $\mathcal{H}_0$ a trivial $\sigma$-algebra.~Moreover $Y$
 is a $(P,\mathbb{H})$-semi-martingale enjoying the p.r.p.~with respect to $\mathbb{H}$
 under an equivalent martingale measure.~The enlarged filtration is defined by
\begin{equation*}\mathbb{G}:=\mathbb{F}\vee\mathbb{H}.\end{equation*}
Next it is proved that there exists $Q\in \mathbb{P}(X, Y, [X,Y],
\mathbb{G})$ such that $(X, Y, [X,Y])$ enjoys the
$(Q,\mathbb{G})$-p.r.p..\\
As a byproduct of this result a \textit{martingale representation
property} is derived: any $(P,\mathbb{G})$-square-integrable
martingale can be uniquely represented up to a constant as sum of
integrals with respect to the martingales $M$, $N$ and $[M,N]$,
where $M$ and $N$ are the martingale parts of $X$ and $Y$,
respectively.~This result is closely related to that presented in
\cite{xue}.~In that paper the author works with the filtrations
generated by two independent, quasi-left continuous
semi-martingales.~Here this regularity for the trajectories is not
required and in place of  the independence of the two
semi-martingales the strong orthogonality of their martingale
parts is assumed.~Indeed, under the assumptions of this section,
independence of $\mathbb{F}$ and $\mathbb{H}$ and strong
orthogonality of $M$ and $N$ are equivalent conditions (see point
i) in Theorem \ref{thm-SEMI-MG-COV-NO-NULLA}).
\\
First, the case when $X$ and $Y$ are $(P,\mathbb{G})$-strongly
orthogonal martingales is considered.~Then the general case
follows by using a key result: the martingale parts $M$ and $N$ of
$X$ and $Y$ enjoy the $(P,\mathbb{F})$-p.r.p.~and the
$(P,\mathbb{H})$-p.r.p.~respectively.
\subsection{The martingale case}
First of all consider the following particular case: the process $A$ in
 decomposition (\ref{semi-martingale}) is identically zero and therefore $X$
 coincides with $M$.~Let $N$ be a square-integrable
 $(P,\mathbb{H})$-martingale on $(\Omega,\mathcal{F}, P)$.~Make the following assumption\vspace{0.5em}\\
\textbf{H1$^\prime$)}\;\textit{$\mathbb{P}(M,\mathbb{F})=\{P_{|\mathcal{F}_T}\},
\ \ \mathbb{P}(N,\mathbb{H})=\{P_{|\mathcal{H}_T}\} $}.\vspace{0.5em}\\
 \noindent It is useful to recall
the definition of strongly orthogonal square-integrable
martingales.
\begin{definition}\label{def-ort-mg} (\cite{Prott})
Two square-integrable $(P,\mathbb{G})$-martingales $U$ and $V$ are
$(P,\mathbb{G})$-strongly orthogonal if their product $UV$ is a
uniformly integrable $(P,\mathbb{G})$-martingale such that
$U_0V_0=0$.
\end{definition}
\begin{lemma}\label{lemma_ort=ind}
Under hypothesis \textbf{H1$^\prime$)}  $\mathcal{F}_T$ and
$\mathcal{H}_T$ are $P$-independent and $M_0N_0=0$ if and only if
 $M$ and  $N$ are  $(P,\mathbb{G})$-martingales and
$(P,\mathbb{G})$-strongly orthogonal.
\end{lemma}
\begin{proof}
The necessary part of the statement is straightforward, since independence together with condition $M_0N_0=0$  implies  strong orthogonality of $M$ and $N$.\\
In order to prove the sufficient condition recall that, since $M$
and $N$ are strongly orthogonal $(P,\mathbb{G})$-martingales, the
process $[M,N]=([M,N]_t)_{t\in[0,T]}$ is  a
$(P,\mathbb{G})$-martingale.~Moreover by \textbf{H1$^\prime$)} it
follows that if $A\in \mathcal{F}_T$ and $B\in \mathcal{H}_T$ then
 \begin{align}\label{rappA}
\mathbb{I}_A=P(A)+\int_0^T\xi^A_sdM_s,\;\;\;\;
\mathbb{I}_B=P(B)+\int_0^T\xi^B_sdN_s,\;\;\;\; P\textrm{-a.s.\;}
\end{align}
for $\xi^A$ and $\xi^B$  in
$\mathcal{L}^2(M,P,\mathbb{F})$ and $\mathcal{L}^2(N,P,\mathbb{H})$
respectively.~The equalities in (\ref{rappA})
imply that $P(A\cap B)$ differs from $P(A)P(B)$ by the expression
\begin{align*}
P(B)E^P\left[\int_0^T\xi^A_sdM_s\right]+P(A)E^P\left[\int_0^T\xi^B_sdN_s\right]+
E^P\left[\int_0^T\xi^A_s\,dM_s\int_0^T\xi^B_s\,dN_s\right].
\end{align*}
The above expression is null.~In fact the
$(P,\mathbb{G})$-martingale property of $M$ and $N$ and the
integrability of the integrands $\xi^A$ and $\xi^B$ imply that
the processes $\int_0^\cdot\xi^A_s\,dM_s$ and
$\int_0^\cdot\xi^B_s\,dN_s$ are centered
$(P,\mathbb{G})$-martingales.~Moreover the stable subspaces
generated by $M$ and $N$ respectively, are
$(P,\mathbb{G})$-strongly orthogonal so that also the product
$\int_0^\cdot\xi^A_s\,dM_s\cdot\int_0^\cdot\xi^B_s\,dN_s$ is a
centered $(P,\mathbb{G})$-martingale (see Lemma 2 and Theorem 36
page 180 in \cite{Prott}).
\end{proof}
\noindent Before stating the main theorem of this section it is
convenient to recall  two general results.
\begin{lemma}\label{lemma-cond.nto}
Let $\mathbb{A}$ and $\mathbb{B}$ be two independent filtrations
on $(\Omega, \mathcal{F}, P)$  and let $U$ and $V$ be two real
processes $\mathbb{A}$-adapted and $\mathbb{B}$-adapted
respectively.~Then for all $0<s<t$
$$
E^P\left[U_tV_t|\mathcal{A}_{s}\vee\mathcal{B}_{s}\right]=E^P\left[U_t|\mathcal{A}_{s}\right]E^P\left[V_t|\mathcal{B}_{s}\right].
$$
\end{lemma}
\begin{lemma}\label{lemma-cad-of-G} Let $\mathbb{A}$ and $\mathbb{B}$ be two filtrations on $(\Omega, \mathcal{F}, P)$
under standard hypotheses of completeness and
right-continuity.~Consider the filtration $\mathbb{D}$ defined at
any time $t$ by $\mathcal{D}_t=\mathcal{A}_t\vee
\mathcal{B}_t$.~If there exists a probability measure $Q$
equivalent to $P$ such that $\mathbb{A}$ and $\mathbb{B}$ are
$Q$-independent, then $\mathbb{D}$ satisfies the standard
hypotheses.
\end{lemma}
\begin{proof}
See Lemma 2.2 in \cite{ame-be-schw03}.
\end{proof}
\begin{theorem}\label{thm-COV-NO-NULLA}
Assume \textbf{H1$^\prime$)} and suppose that $M$ and $N$ are
$(P,\mathbb{G})$-strongly orthogonal martingales.~Then
$\mathbb{G}$ is a standard filtration,
$\mathbb{P}\left((M,N,[M,N]),\mathbb{G}\right)=\{P|_{\mathcal{G}_T}\}$
and the following decomposition holds
\begin{equation}\label{eq-decomposition}
L_0^2(\Omega,\mathcal{G}_T, P)=K^2(\Omega,\mathbb{G},P,M)\oplus
K^2(\Omega,\mathbb{G},P,N)\oplus K^2(\Omega,\mathbb{G},P,[M,N]).
\end{equation}
\end{theorem}
\begin{proof}
The proof will
be done in three steps: \\
(i) the first goal is to prove the $(P,\mathbb{G})$-p.r.p.~for
$(M, N, [M,N])$; \\
(ii) as a second point the following key result is proved: $[M,N]$
is $(P,\mathbb{G})$-strongly orthogonal to $M$ and to
$N$;\\
(iii) finally points (i) and (ii) allow to derive  decomposition
(\ref{eq-decomposition}).
\begin{itemize}
    \item [(i)] $(P,\mathbb{G})$-strong orthogonality of $M$ and $N$ and
Lemma \ref{lemma_ort=ind} provide the $P$-independence of
$\mathcal{F}_T$ and $\mathcal{H}_T$, so that the standard
conditions for $\mathbb{G}$ immediately follow by Lemma
\ref{lemma-cad-of-G}.~The $(P,\mathbb{G})$-p.r.p.~for
$(M,N,[M,N])$ is achieved  by proving that
$\mathbb{P}((M,N,[M,N]),\mathbb{G})=\{P\}$, or, equivalently, that
for any $Q\in\mathbb{P}((M,N,[M,N]),\mathbb{G})$,  $P$ and $Q$
coincide on the $\pi$-system $$\{A\cap B, \ A\in\mathcal{F}_{T}, \
B\in\mathcal{H}_{T}\},$$ which generates $\mathcal{G}_T$.~To this
end, note that equalities in (\ref{rappA}) hold under $Q$ so that
$Q(A\cap B)$ differs from $P(A)P(B)$ by the expression
\begin{align}
P(B)E^Q\left[\int_0^T\xi^A_sdM_s\right]+P(A)E^Q\left[\int_0^T\xi^B_sdN_s\right]+
E^Q\left[\int_0^T\xi^A_s\,dM_s\int_0^T\xi^B_s\,dN_s\right].
\end{align}
The above expression is null.~In fact \textbf{H1$^\prime$)}
implies $Q|_{\mathcal{F}_T}=P|_{\mathcal{F}_T}$ and
$Q|_{\mathcal{H}_T}=P|_{\mathcal{H}_T}$ and this in turn implies
that $\int_0^\cdot\xi^A_sdM_s$ and $\int_0^\cdot\xi^B_sdN_s$ are
centered $(Q,\mathbb{G})$-martingales.~Moreover, by definition of
$Q$, $[M,N]$ is a $(Q,\mathbb{G})$-martingale so that $MN$ is a
$(Q,\mathbb{G})$-martingale and therefore  $M$ and  $N$ are
$(Q,\mathbb{G})$-strongly orthogonal martingales, so that also the
product $\int_0^\cdot\xi^A_s\,dM_s\cdot\int_0^\cdot\xi^B_s\,dN_s$
is a centered $(Q,\mathbb{G})$-martingale.\bigskip\\
    \item [(ii)]
$[M,N]$ is $(P,\mathbb{G})$-strongly orthogonal to the
$(P,\mathbb{G})$-martingales $M$ and $N$, if and only if
$\left[M,[M,N]\right]$ and $\left[N,[M,N]\right]$ are uniformly
integrable $(P,\mathbb{G})$-martingales.\\Recall that
\begin{equation}\label{quadratic_covariation}
[M,N]_t=\langle M^c,N^c\rangle_t+\sum_{s\le t}\Delta M_s\Delta
N_s,
\end{equation}
where $M^c$ and $N^c$ are the continuous martingale part of $M$
and $N$ respectively. By Lemma \ref{lemma_ort=ind} $M^c$ and $N^c$
are independent $(P,\mathbb{G})$-martingales so that $\langle
M^c,N^c\rangle\equiv 0$, since by definition  $\langle
M^c,N^c\rangle$ is the unique $\mathbb{G}$-predictable process
with finite variation such that $M^c\,N^c-\langle M^c,N^c\rangle$
is $\mathbb{G}$-local martingale equal to 0 at time 0 (see
Subsection 9.3.2. in \cite{Jean-yor-chesney}).~Therefore
\begin{equation}\label{eq-covar}
[M,N]_t=\sum_{s\le t}\Delta M_s\Delta N_s.
\end{equation}
As a consequence
$$
\left[M,[M,N]\right]_t=\sum_{s\le t}(\Delta M_s)^2\Delta N_s.
$$
Then for $u\le t$ one has
\begin{align*}
&E^P\left[\left[M,[M,N]\right]_t|\mathcal{G}_u\right]\\&=E^P\left[\sum_{s\le
u}(\Delta M_s)^2\Delta N_s|\mathcal{G}_u\right]+E^P\left[\sum_{u<s\le t}(\Delta M_s)^2\Delta N_s|\mathcal{G}_u\right]\\
&=\left[M,[M,N]\right]_u+\sum_{u<s\le t}E^P\left[(\Delta
M_s)^2\Delta
N_s|\mathcal{G}_u\right]\\
&=\left[M,[M,N]\right]_u+\sum_{u<s\le t}E^P\left[(\Delta
M_s)^2|\mathcal{F}_u\right]E^P\left[\Delta
N_s|\mathcal{H}_u\right],
\end{align*}
where the last equality follows by Lemma \ref{lemma-cond.nto}
since $\mathbb{F}$ and $\mathbb{H}$ are $P$-independent.\\Then the
martingale property for $[M,[M,N]]$ follows by observing that
$E^P\left[\Delta N_s|\mathcal{H}_u\right]=0$, for any
$s>u$.~Finally
$\left[M,[M,N]\right]$ is uniformly integrable, since it is a $(P,\mathbb{G})$-regular martingale.\\
Analogously one gets that $[M,N]$ is $(P,\mathbb{G})$-strongly
orthogonal to $N$.\bigskip\\
    \item  [(iii)]
By point (i) it follows that (see (\ref{natural-completeness-2}))
$$L_0^2(\Omega,\mathcal{G}_T,
P)=K^2\left(\Omega,\mathbb{G},P,(M,N,[M,N])\right)$$ or
equivalently that for each $H$ in $L_0^2(\Omega,\mathcal{G}_T, P)$
there exist $\gamma^H$ in $\mathcal{L}^2(M,P,\mathbb{G})$,
$\kappa^H$ in $\mathcal{L}^2(N,P,\mathbb{G})$ and $\phi^H$ in
$\mathcal{L}^2([M,N],P,\mathbb{G})$, such that $P$-a.s.
 \begin{align}\label{rappr}
H=\int_0^T\gamma^H_sdM_s+\int_0^T\kappa^H_sdN_s+\int_0^T\phi^H_sd[M,N]_s.
\end{align}
This equality and point (ii)  entail
$$K^2(\Omega,\mathbb{G},P,[M,N])=\Big(K^2(\Omega,\mathbb{G},P,(M,N))\Big)^\perp$$
In fact the $(P,\mathbb{G})$-strong orthogonality of $[M,N]$ to
$M$ and to $N$ is equivalent to the orthogonality of any random
variable of the form $\int_0^T\phi_sd[M,N]_s$, with $\phi$ in
$\mathcal{L}^2([M,N],P,\mathbb{G})$, to random variables of the
form $\int_0^T\gamma_sdM_s+\int_0^T\kappa_sdN_s$, with  $\gamma$
in $\mathcal{L}^2(M,P,\mathbb{G})$ and $\kappa$ in
$\mathcal{L}^2(N,P,\mathbb{G})$
 (see Lemma
2 and Theorem 36 page 180 in \cite{Prott}), so that
$$K^2(\Omega,\mathbb{G},P,[M,N])\subset\Big(K^2(\Omega,\mathbb{G},P,(M,N))\Big)^\perp.$$
At the same time, by representation (\ref{rappr}),  any element of
$\Big(K^2(\Omega,\mathbb{G},P,(M,N))\Big)^\perp$ is of the form
$\int_0^T\phi_sd[M,N]_s$ so that
$$\Big(K^2(\Omega,\mathbb{G},P,(M,N))\Big)^\perp\subset K^2(\Omega,\mathbb{G},P,[M,N]).$$
Finally, by the $(P,\mathbb{G})$-strong orthogonality of $M$ and
$N$,
$$K^2(\Omega,\mathbb{G},P,(M,N))=K^2(\Omega,\mathbb{G},P,M)\oplus
K^2(\Omega,\mathbb{G},P,N).$$
\end{itemize}
\end{proof}
\begin{remark}\label{multid-extens3}
\noindent Lemma \ref{lemma_ort=ind} and Theorem
\ref{thm-COV-NO-NULLA} can be extended to the case when $M$ and
$N$ take values in $\mathbb{R}^m$ and $\mathbb{R}^n$ respectively
and $M^i$ and $N^j$ are $(P,\mathbb{G})$-strongly orthogonal
martingales for all $i=1,\ldots , m,\;j=1,\ldots , n$.

\noindent For the sake of completeness, recall that  an
$\mathbb{R}^m$-valued square integrable martingale
 $M$  enjoys the $(P,\mathbb{F})$-p.r.p. if each $H$ in
$L^2(\Omega,
 \mathcal{F}_T, P)$ can be represented as vector stochastic integral that is
\begin{align*}
H=H_0+\int_0^T\xi^HdM,
\end{align*}
with $H_0\in \mathcal{F}_0$  and $\xi^H=(\xi^H_1,...,\xi^H_m)$ an
$m$-dimensional $\mathbb{F}$-predictable  process  such that
 such that
\begin{equation*}E^P\left[\sum_{i,j}\int_0^T\xi^H_i(t)\xi^H_j(t)\,d[M^i,M^j]_t\right]<+\infty,\end{equation*}
(see \cite{cha-stri94}).
\end{remark}
\begin{example}\label{ex-prp_ortogonal_mg}
On a probability space $(\Omega, \mathcal{F},P)$ let $M$ be the
process defined at time $t$ in $[0,T]$ by
\begin{equation}\label{def:example-Mmg}M_t:=B_t+Z\,\mathbb{I}_{\{t\geq
t_0\}},\end{equation} where $B$ is a standard Brownian motion and
$Z$ a binary centered $\mathcal{F}$-measurable random variable
independent of $B$.~Write for shortness
$$H_t:=Z\,\mathbb{I}_{\{t\geq t_0\}}.$$Following a similar
procedure as in Section 9.5.2 in \cite{Jean-yor-chesney}, it can
be proved that $M$ enjoys the p.r.p.~with respect to its natural
filtration $\mathbb{F}^M=\mathbb{F}^B\vee\mathbb{F}^H$ (under
$P$).\bigskip\\In fact consider the Dol\'eans-Dade exponentials
$$\mathcal{E}^{1,\varphi}_\cdot:=\mathcal{E}\Big(\int_0^\cdot\varphi_sdB_s\Big),\;\;\;\;\mathcal{E}^{2,\varphi}_\cdot:=\mathcal{E}\Big(\int_0^\cdot\varphi_sdH_s\Big),$$
\noindent with $\varphi\in L^2([0,T])$.~Both are square-integrable
random
 processes and in particular $\mathcal{E}\Big(\int_0^t\varphi_sdH_s\Big)=1+\varphi_{t_0}Z\,\mathbb{I}_{\{t\geq t_0\}}.$
Moreover, by the product formula
   it follows
\begin{equation}\label{rappr-DD}\mathcal{E}^{1,\varphi}_T\mathcal{E}^{2,\varphi}_T=1+\int_0^T\mathcal{E}^{1,\varphi}_t\mathcal{E}^{2,\varphi}_{t^-}\varphi_t
dM_t.\end{equation}
The set
$\{\mathcal{E}^{1,\varphi}_T\mathcal{E}^{2,\varphi}_T, \varphi\in
L^2([0,T])\}$ is a total set for $L^2(\Omega,\mathcal{F}^M_T,P)$
and therefore (\ref{rappr-DD}) is
equivalent to the p.r.p.~for $M$ with respect to  $\mathbb{F}^M$ (under $P$).\bigskip\\
Let now $N$  be defined by
\begin{equation}\label{def:example-Nmg}N_t:=\tilde{B}_t+U\mathbb{I}_{\{t\geq t_0\}}\end{equation}
and assume $\tilde{B}$ and $U$ independent of $B$ and $Z$.\\Then
$\mathbb{F}^M$ and $\mathbb{F}^N$ are independent filtrations and
$M$ and $N$ are square-integrable
$\mathbb{F}^M\vee\mathbb{F}^N$-strongly orthogonal  martingales
each of them enjoying the p.r.p.~with respect to its natural
filtration (under $P$).~Moreover the covariation process $[M,N]$
at time $t$ satisfies
$$[M,N]_t=ZU\mathbb{I}_{\{t\geq t_0\}}.$$
Theorem \ref{thm-COV-NO-NULLA} applies with
$\mathbb{F}=\mathbb{F}^M$ and $\mathbb{H}=\mathbb{F}^N$ so that
any $K$ in $L^2(\Omega,
 \mathcal{F}^M_T\vee\mathcal{F}^N_T,P)$ can be represented
 $P$-a.s.~as
\begin{align}\label{eq-rapp_semi-mg ortog}
K=K_0+\int_0^T\gamma^K_tdM_t+\int_0^T\kappa^K_tdN_t+\Phi^KZU,
\end{align}
where  $\gamma^K$ and $\kappa^K$ belong to $\mathcal{L}^2(M,P,
\mathbb{F}^M\vee\mathbb{F}^N)$ and $\mathcal{L}^2(N,P,
\mathbb{F}^M\vee\mathbb{F}^N)$ respectively and $\Phi^H$ is a
square-integrable random variable
$\mathcal{F}^M_{t_0^-}\vee\mathcal{F}^N_{t_0^-}$-measurable.
\end{example}
\subsection{The semi-martingale case}\label{subsec:semim}
\noindent In the general setup of this section let $Y$ be in the
 space of semi-martingales $\mathcal{S}^2(P,
\mathbb{H})$ with canonical decomposition
\begin{equation}\label{eq_semimgY}
 Y=Y_0+N+D.
 \end{equation}
Assume that $\mathbb{P}(Y,\mathbb{H})$ is a singleton and more precisely $\mathbb{P}(Y,\mathbb{H})=\{P^Y\}$ so that
$\mathcal{H}_0$ is $P$-trivial and the semi-martingale $Y$
enjoys the p.r.p.~with respect to $\mathbb{H}$ under $P^Y$.\vspace{0.5em}\\It is convenient to write the main assumptions on $X$ and $Y$ as a unique condition:\vspace{0.5em}\\
\textbf{H1$^{\prime\prime}$)}
\;\textit{$\mathbb{P}(X,\mathbb{F})=\{P^X\},$}
\;\textit{$\mathbb{P}(Y,\mathbb{H})=\{P^Y\}$}.\vspace{1em}\\
Some other assumptions will be considered.~Denote by
$\overline{K^X(\mathbb{F})}$ the closure in $L^1(\Omega,
 \mathcal{F}_T, P)$ of the set
$\left\{\int_0^T\xi_sdX_s,\;
\xi\;\;\mathbb{F}\textrm{-predictable,
simple\,and\,bounded}\right\}$ and by
$L_+^1(\Omega,\mathcal{F}_T,P)$
 the set of the non-negative integrable random variables.~Define analogously $\overline{K^Y(\mathbb{H})}$ and $L_+^1(\Omega,\mathcal{H}_T,P)$.\vspace{0.5em}\\
\textbf{H3)}\;  $\overline{K^X(\mathbb{F})}\cap
L_+^1(\Omega,\mathcal{F}_T,P)
=\{0\},\;\;\;\;\;\overline{K^Y(\mathbb{H})}\cap
L_+^1(\Omega,\mathcal{H}_T,P) =\{0\}.$
\vspace{0.5em}\\
Assumption \textbf{H3)} together with condition
(\ref{eq-cond_struttura}) and its analogous for $Y$ provide  the
so-called \textit{structure condition} for $X$ and $Y$ (see
Theorem 8 in \cite{ans_str92}), that is the existence of a
predictable process $\alpha$ $P$-a.s.~in the space $L^2([0,T],
\mathcal{B}([0,T]), d\langle M\rangle_t)$ and of a predictable
process $\delta$ $P$-a.s.~in the space $L^2([0,T],
\mathcal{B}([0,T]), d\langle N\rangle_t)$ such that
\begin{equation}\label{struct-cond}A_t=\int_0^t\alpha_s\,d\langle M\rangle_s,\;\;\;\;\;D_t=\int_0^t\delta_s\,d\langle N \rangle_s.\end{equation}
\begin{remark}\label{rem-weak-condition} It is to note that in order to get the structure condition for $X$ hypothesis \textbf{H3)} could be
omitted by assuming $dP^X/dP|_{\mathcal{F}_T}$ in
$L^2_{loc}(\Omega,\mathcal{F},P)$ (see Proposition 4 in
\cite{schweizer92}).
\end{remark}
\noindent Finally two hypotheses assure regularity to the
Dol\'eans-Dade exponentials
$\mathcal{E}(-\int_0^\cdot\alpha_sdM_s)$ and
$\mathcal{E}(-\int_0^\cdot\delta_sdN_s)$.\vspace{1em}\\
\textbf{H4)}\;$\alpha\Delta M<1$, $P$-a.s.~,\;\;$\delta\Delta N<1$, $P$-a.s.~,\\
\textbf{H5)}\;
\begin{align*}&E^P\left[\exp\left\{\frac12\int_0^T\alpha^2_t\,d\langle M^c\rangle_t+\int_0^T\alpha^2_t\,d\langle M^d\rangle_t\right\}\right]<+\infty,\\&E^P\left[\exp\left\{\frac12\int_0^T\delta^2_t\,d\langle N^c\rangle_t+\int_0^T\delta^2_t\,d\langle N^d\rangle
_t\right\}\right]<+\infty.\end{align*}
\vspace{1em}\\In particular \textbf{H4)} implies that the $(P,\mathbb{F})$-local martingale
$\mathcal{E}(-\int_0^\cdot\alpha_sdM_s)$
 is strictly positive and \textbf{H5)} implies that it is a
   $(P,\mathbb{F})$-martingale (see Theorem 9 in \cite{pro-sh08} for further details).
\begin{definition}(\cite{an-str93})\label{def-min-mart-meas}
A  measure $Q$ in  $\mathbb{P}(X,\mathbb{F})$  is a minimal
martingale measure for $X$ if any $(P,\mathbb{F})$-local
martingale $Z$ such that $ZM$ is a $(P,\mathbb{F})$-local
martingale is a $(Q,\mathbb{F})$-local martingale.
\end{definition}
\begin{lemma}\label{lemma-minimality}
Under the previous assumptions $P^X$ is the minimal martingale
measure for $X$.
\end{lemma}
\begin{proof}
By Proposition 3.1 in  \cite{an-str93}
$\mathcal{E}(-\int_0^\cdot\alpha_sdM_s)$ coincides with the
derivative of the minimal martingale measure for $X$.~The assumed
uniqueness of the equivalent martingale measure gives the result.
\end{proof}
\noindent Before stating next result, it is useful to recall the
general result known as Yoeurp's lemma.
\begin{lemma}\label{lemma-Yoeurp}(\cite{del-me-b})
 Let M be a local martingale null at time zero and A a c\`adl\`ag process with
finite variation on every compact set.~If A is predictable then
$[M,A]$ is a local martingale.
\end{lemma}
\begin{proposition}\label{prop-pred-mart-rep}
Let \textbf{H1$^{\prime\prime}$)}, \textbf{H3)}, \textbf{H4)} and
 \textbf{H5)} be verified.~Then $M$ enjoys the
p.r.p.~with respect to the filtration $\mathbb{F}$
 under $P$.
\end{proposition}
\begin{proof}
Assume that $K^2(\Omega,\mathbb{F},P,M)^\perp\neq\{0\}$, that is
there  exists a non trivial centered random variable $V\in
K^2(\Omega,\mathbb{F},P,M)^\perp$ such that the
$(P,\mathbb{F})$-martingale $(V_t)_{t\in [0,T]}$ defined by
$V_t:=E^P[V|\mathcal{F}_t]$ is $(P,\mathbb{F})$-strongly
orthogonal to $M$.~By Lemma \ref{lemma-minimality} $(V_t)_{t\in
[0,T]}$ is a $(P^X,\mathbb{F})$-local martingale.~Then by
\textbf{H1$^{\prime\prime}$)}
 there exists a predictable process $\xi$ such that, for all $t\in[0,T]$,
$P^X$-a.s.~
\begin{equation}\label{eq:repr-ort-mart}
V_t=\int_0^t\xi_udX_u.
\end{equation}
As a consequence, since the covariation processes are invariant
under an equivalent change of measure, $P^X$-a.s.~and $P$-a.s.~
\begin{equation*}
[V,X]_t=\int_0^t\xi_ud[X]_u.
\end{equation*}
Under $P$ the process on the left hand side in the previous
equality is a $(P,\mathbb{F})$-local martingale.~In fact under $P$
one has $[V,X]=[V,M]+[V,A]$ and by construction $[V,M]$ is a
$(P,\mathbb{F})$-martingale and by Lemma \ref {lemma-Yoeurp} the
process
 $[V,A]$ is a  $(P,\mathbb{F})$-local martingale.
More precisely $[V,X]$ is a $(P,\mathbb{F})$-martingale since for
all $t$ in $[0,T]$
$$
\left|\left[V,X\right]_t\right|\le [V+X]_T+[V-X]_T$$ and the right
hand side of the above inequality is integrable by the assumption
(\ref{eq-int-cond-X}) on $X$ and the construction of $(V_t)_{t\in
[0,T]}$ .~Then also the process $(\int_0^t\xi_ud[X]_u)_{t\in
[0,T]}$ is a $(P,\mathbb{F})$-martingale so that for all $0<s<t<T$
and $B\in \mathcal{F}_s$ it holds $$ E^P\left[I_B\int_s^t\xi_u
d[X]_u\right]=0$$ and therefore
$$\xi=0,\,P(d\omega)d[X]_t(\omega)\textrm{-a.s}.~.$$ Finally by the equivalence
between $P$ and $P^X$ and the invariance of $[X]$ under equivalent
change of measure
$$\xi=0,\,P^X(d\omega)d[X]_t(\omega)\textrm{-a.s}.~.$$ As a
consequence by Ito isometry
\begin{equation}\label{mg-of-V}\int_0^t\xi_u d
X_u=0,\;\;P^X\textrm{-a.s.}\end{equation}
that is, by equality (\ref{eq:repr-ort-mart}),
 $V_T=0,\,P^X\textrm{-a.s.}$ which contradicts
the supposed non-triviality of $V$.
\end{proof}
\begin{remark}
In the particular case when $X$ is the solution of an Ito equation
with unique weak solution and $\mathbb{F}=\mathbb{F}^X$ the above
result immediately follows  from Theorem 9.5.4.2 in
\cite{Jean-yor-chesney}.
\end{remark}
\begin{theorem}\label{thm-SEMI-MG-COV-NO-NULLA}
Assume \textbf{H1$^{\prime\prime}$)}, \textbf{H3)},\textbf{H4)},
\textbf{H5)} and suppose that $M$ and $N$ are
$(P,\mathbb{G})$-strongly orthogonal martingales.~Then
\begin{itemize}
\item[i)] $\mathcal{F}_T$ and $\mathcal{H}_T$ are
$P$-independent, $\mathbb{G}$ fulfills the standard hypotheses and
every $W$ in $\mathcal{M}^2(P,\mathbb{G})$ can be uniquely
represented as
\begin{align*}
W_t=W_0+\int_0^t\gamma^W_sdM_s+\int_0^t\kappa^W_sdN_s+\int_0^t\phi^W_sd[M,N]_s,\;\;\;P\textrm{-a.s.}~
\end{align*}
with $\gamma^W$ in $\mathcal{L}^2(M, P, \mathbb{G})$, $\kappa^W$
in $\mathcal{L}^2(N, P, \mathbb{G})$ and $\phi^W$ in
$\mathcal{L}^2([M,N], P, \mathbb{G})$;
\item[ii)] there exists a probability measure $Q$ on $(\Omega,
\mathcal{G}_T)$ such that $(X,Y,[X,Y])$ enjoys the p.r.p.~with
respect to $\mathbb{G}$ under $Q$.~More precisely every
 $Z$ in $\mathcal{M}^2(Q,\mathbb{G})$ can be uniquely represented as
\begin{align*}
Z_t=Z_0+\int_0^t\eta^Z_sdX_s+\int_0^t\theta^Z_sdY_s+\int_0^t\zeta^Z_sd[X,Y]_s\;\;\;Q\textrm{-a.s.},
\end{align*}
with $\eta^Z$ in $\mathcal{L}^2(X, Q, \mathbb{G})$, $\theta^Z$ in
$\mathcal{L}^2(Y, Q, \mathbb{G})$ and $\zeta^Z$ in
$\mathcal{L}^2([X,Y], Q, \mathbb{G})$.
\end{itemize}
\end{theorem}
\begin{proof}
By the previous proposition and its analogous for $N$, the
martingales $M$ and $N$ satisfy condition \textbf{H1$^{\prime}$)}
so that Theorem \ref{thm-COV-NO-NULLA} applies and the first
statement is proved.\\Define
 $Q$ on $(\Omega,\mathcal{G}_T)$ by
$$\frac{dQ}{dP}:=L^X\cdot L^Y$$
where
$$L^X:=\frac{dP^X}{dP|_{\mathcal{F}_T}},\;\;\;\;\;\;L^Y:=\frac{dP^Y}{dP|_{\mathcal{H}_T}}.$$
The definition is well-posed since by point i) $L^X\cdot L^Y$ is
in $L^1(\Omega, P,\mathcal{G}_T)$.~$L^X$ and $L^Y$ are strictly
positive and therefore  $Q$ and $P|_{\mathcal{G}_T}$ are
equivalent measures.~Moreover for all $A$ in $\mathcal{F}_T$ and
$B$ in $\mathcal{H}_T$ it holds
$$
    Q(A\cap B)=E^P[\mathbb{I}_A\,L^X]\,E^P[\mathbb{I}_B\,
    L^Y],
$$
since $\mathcal{F}_T$ and $\mathcal{H}_T$ are independent under
$P$.~Using the equalities $E^P[L^X]$ =1=
$E^P[L^Y]$ one immediately gets the $Q$-independence of $\mathcal{F}_T$ and $\mathcal{H}_T$.\\
Finally $X$ is a $(Q,\mathbb{F})$-martingale since
$Q|_{\mathbb{F}}=P^X$ and it is also a $(Q,\mathbb{G})$-martingale
 by the $Q$-independence of $\mathbb{F}$ and $\mathbb{H}$.~Analogously it can be shown that  $Y$ is a
$(Q,\mathbb{G})$-martingale.~Since $X_0$ and $Y_0$ are constants,
the second statement follows by applying Theorem
\ref{thm-COV-NO-NULLA}
  to the  $(Q,\mathbb{G})$-strongly orthogonal martingales $X-X_0$ and $Y-Y_0$.
\end{proof}
\begin{remark}\label{rem:con-H2}
Under the hypotheses of the above theorem \textbf{H2)} is
verified.~In fact $Q$ belongs to $\mathbb{P}(X,\mathbb{G})$.~This
immediately follows recalling that $Q$ is equivalent to
$P|_{\mathcal{G}_T}$, coincides with $P^X$ on $\mathcal{F}_T$ and
decouples $\mathbb{F}$ and $\mathbb{H}$ in such a way that the
immersion property is verified.
\end{remark}
\begin{remark}\label{rem:R}
The $(P,\mathbb{G})$-strong orthogonality of $M$ and $N$ in
Theorem \ref{thm-SEMI-MG-COV-NO-NULLA} can be weakened assuming
the existence of a measure $P^*$ equivalent to $P$ such that $M$
and $N$ are $(P^* ,\mathbb{G})$-strongly
orthogonal~martingales.~In this case, $P$ in i) has to be replaced
by $P^*$ and $Q$ in ii) has to be changed into the measure defined
by
$\frac{dP^X}{dP^*|_{\mathcal{F}_T}}\frac{dP^Y}{dP^*|_{\mathcal{H}_T}}\,dP^*$.~Note
 that by Lemma \ref{lemma_ort=ind} $\mathcal{F}_T$ and $\mathcal{H}_T$ are $P^*$-independent.~On the other hand if there exists a measure $R$ equivalent to $P$
 such that $\mathcal{F}_T$ and $\mathcal{H}_T$ are $R$-independent, then $\frac{dP|_{\mathcal{F}_T}}{dR|_{\mathcal{F}_T}}\frac{dP|_{\mathcal{H}_T}}{dR|_{\mathcal{H}_T}}\,dR$
 defines a measure with the same properties as the previous $P^*$.\\
Finally when $\mathbb{H}$ coincides with $\mathbb{F}^Y$ and $Y$ is
$P$-independent of $\mathbb{F}$, then the above conditions are
verified with $P^*=R=P$.
\end{remark}
\begin{cor}\label{rapp semimg ortog+COV_NULLA}
Under the hypotheses of the previous theorem, $X$ and $Y$ verify
$[X,Y]_t\equiv 0$ $P$-a.s.~if and only if the
$(P,\mathbb{G})$-semi-martingale $(X, Y)$
 enjoys the p.r.p.~with respect to $\mathbb{G}$.
\end{cor}
\begin{remark}
Note that the vanishing of the quadratic covariation as sufficient
condition for the p.r.p.~of a pair of orthogonal martingales each
 enjoying the p.r.p.~with respect to its own filtration is already known (see
\cite{Lokka}).
\end{remark}
\noindent Corollary \ref{rapp semimg ortog+COV_NULLA} implies that
if $X$ and $Y$ in Theorem \ref{thm-SEMI-MG-COV-NO-NULLA} are
quasi-left continuous then the pair $(X,Y)$ enjoys the p.r.p.~with
respect to $\mathbb{G}$ under $Q$.~In fact a semi-martingale is
quasi-left continuous if and only its jumps times are totally
inaccessible so that, using
\begin{equation}
[X,Y]_t=\langle M^c,N^c\rangle_t+\sum_{s\le t}\Delta X_s\Delta Y_s
\end{equation}
and the $P$-independence of $\mathbb{F}$ and $\mathbb{H}$, it
follows that $[X,Y]\equiv 0$.~In particular the second addend at
the right hand side is zero, since two independent processes
cannot have common inaccessible
jump times with positive probability.\\
In conclusion, under the hypotheses of Theorem
\ref{thm-SEMI-MG-COV-NO-NULLA}, $[X,Y]\neq 0$ if and only if $X$
and $Y$ share accessible jump times with positive
probability.\bigskip\\
\noindent Along the lines of Example \ref{ex-prp_ortogonal_mg} it
is easy to construct a pair $(X,Y)$ of semi-martingales satisfying
\textbf{H1$^{\prime\prime}$)}, \textbf{H4)} and \textbf{H5)} with
strongly orthogonal
 martingale parts and  covariation process not identically
 zero.~In this case the structure condition is given by the model so that assumption \textbf{H3)} is superfluous.
 \begin{example}
Let $M$ be defined as in (\ref{def:example-Mmg}).~Note that
$$\langle M\rangle_t=t+E^P[Z^2\mid \mathcal{F}_{t-}]\,\mathbb{I}_{\{t\geq t_0\}}.$$
Consider a continuous function $\left(\alpha_t\right)_{t\in[0,T]}$
in $L^2([0,T])$ such that $\alpha_{t_0}=0$.~Then
$\int_0^T\alpha_sd\langle M\rangle_s=\int_0^T\alpha_sds$.~Define
the $(P,\mathbb{F}^M)$-semi-martingale $X$ by
$$
X_t:=B_t+\int_0^t\alpha_{{s}}\,ds+\,Z\,\mathbb{I}_{\{t\geq t_0\}}.
$$
$X$ satisfies the structure condition with martingale part equal
to $M$ as well as conditions \textbf{H4)} and
\textbf{H5)}.~Moreover $X$ enjoys the p.r.p.~with respect to
$\mathbb{F}^M$ under the measure $P^X$ defined by
$$dP^X:=L\,dP|_{\mathcal{F}^M_T}$$ with
$L:=\mathcal{E}(-\int_0^T\alpha_sdB_s)$.~The last statement
follows by Example \ref{ex-prp_ortogonal_mg} considering that
$\mathbb{F}^M=\mathbb{F}^B\vee\mathbb{F}^H$, the process
$\big(B_t+\int_0^t\alpha_{{s}}\,ds\big)_{t\in[0,T]}$ under $P^X$
is a standard Brownian motion independent of $Z$ and its natural
filtration coincides with $\mathbb{F}^B$ since $\alpha$ is
deterministic.~Analogously, define a second semi-martingale
$$
Y_t:=\tilde{B}_t+\int_0^t\delta_{{s}}\,ds+\,U\,\mathbb{I}_{\{t\geq
t_0\}}
$$
with $N$ as in (\ref{def:example-Nmg}) and the function
$\left(\delta_t\right)_{t\in[0,T]}$~continuous, in  $L^2([0,T])$
and such that $\delta_{t_0}=0$.~Then
$[X,Y]_t=[M,N]_t=ZU\mathbb{I}_{\{t\geq t_0\}}$.~In conclusion
Theorem \ref{thm-SEMI-MG-COV-NO-NULLA} applies to show that the
triplet $(X,Y, ZU\,\mathbb{I}_{\{\cdot\geq t_0\}})$ enjoys the
p.r.p.~with respect to $\mathbb{F}^M\vee\mathbb{F}^N$.~The pair
$(X,Y)$ instead admits infinite equivalent martingale
measures\footnote{e.g.~if $Z$ takes values in $(-z,z)$ and $U$
takes values in $(-u,u)$, then to any choice of $P(Z=z,U=u)$ in
$(0,1/2)$ it corresponds a different joint law for $(Z,U,ZU)$
which preserves the law of $(Z,U)$}.
\end{example}
\begin{cor}
Under the hypotheses of Theorem \ref{thm-SEMI-MG-COV-NO-NULLA} the
p.r.p.~for $X$  is not preserved by the enlargement of filtration
from $\mathbb{F}$ to $\mathbb{G}$.
\end{cor}
\begin{proof}
$Q|_{\mathcal{F}_T}=P^X$ so that $X$ enjoys the p.r.p.~with
respect to $\mathbb{F}$ under $Q|_{\mathcal{F}_T}$.~Moreover $Q$ belongs
 to $\mathbb{P}(X,\mathbb{G})$ but this set is not a
singleton.~In fact let $H$ be any non trivial set in
$\mathcal{H}_T$ and set $Z:=\frac{I_H}{2Q(H)}+
\frac{I_{H^c}}{2Q(H^c)}.$ Consider on $\mathcal{G}_T$ the measure
$\widetilde{Q}$ defined as the unique extension of
$\widetilde{Q}(A\cap B):=Q(A)E^Q[Z\mathbb{I}_B]$, with
$A\in\mathcal{F}_T$ and $B\in\mathcal{H}_T$.~Then it is easy to
see that
 $\widetilde{Q}$ belongs to $\mathbb{P}(X,\mathbb{G})$.
\end{proof}
\section{Conclusions}\label{sec:conclusions}
In all the paper a fundamental role is played by the existence of
an equivalent martingale measure for $X$ with respect to
$\mathbb{G}$.~In the framework of mathematical finance this
condition implies the financial market with the enlarged
information $\mathbb{G}$ to be free of arbitrage (see
\cite{de-scha94}).~This property is given as hypothesis in Section
\ref{sec:inf=min} while in Section \ref{sec:adding-semimg} it
derives from the assumptions.~As far as the results in  Section
\ref{sec:adding-semimg} are concerned, note that it is common
practice to complete the market by adding new components to the
discounted asset price vector (as recent contributions in this
sense see \cite{Davis-obloy}, \cite{Jacod-prott2010} and
\cite{co-nu-sc05}).~Theorem \ref{thm-SEMI-MG-COV-NO-NULLA}
suggests in particular how to complete the market with discounted
asset price $X$ when the available information $\mathbb{F}$ is
enlarged by the observation of an independent semi-martingale $Y$:
it is enough to assume the processes $Y$ and $[X,Y]$ as the
discounted prices of two new
assets.\vspace{0.3em}\\
Amendinger, Becherer and Schweizer in \cite{ame-be-schw03}, in the
framework of maximization of utilities, generalizing a result in
\cite{ame}, show that the p.r.p.~of $X$ transfers from
$\mathbb{F}$ to $\mathbb{G}=\mathbb{F}\vee\sigma(G)$ with $G$ a
$\mathcal{F}$-measurable random variable.~The basic assumption in
that paper is condition\\
 \textbf{D)} \textit{there exists a
probability measure $R$ on $\mathcal{F}$ equivalent to $P$ such
that $\mathcal{F}_T$ and $\sigma(G)$ are $R$-independent}.\vspace{0.2em}\\
Grorud and Pontier in \cite{gro_po99} study existence and
characterization of the risk-neutral probabilities for an insider
trader with initial information.~In particular under hypothesis
\textbf{H1)} they show the equivalence between conditions \textbf{H2)} and \textbf{D)}.\vspace{0.3em}\\
\noindent
In the present paper, in the setting of Section
\ref{sec:adding-semimg}, the role of condition \textbf{D)} is
played by one of the equivalent conditions discussed in Remark
\ref{rem:R}.~It is to note that condition \textbf{D)} in the
setting of initial enlargement of filtrations allows to prove that
p.r.p.~is preserved, whereas its generalization to the progressive
enlargement in Section \ref{sec:adding-semimg} produces the loss
of p.r.p..~Indeed a general result can be announced.
 \begin{proposition}\label{prop: GrorudPontier} Assume \textbf{H1)} and $\mathbb{G}:=\mathbb{F}\vee\mathbb{K}$ where
 $\mathbb{K}$ is a filtration on $(\Omega,\mathcal{F},P)$ satisfying the standard conditions
 and such that $\mathcal{K}_t\neq \mathcal{K}_0$ for some $t$ in $(0,T]$.~
 Let $R$ be a probability measure on $\mathcal{F}$ equivalent to $P$ such that $\mathcal{F}_T$ and $\mathcal{K}_T$ are  $R$-independent.~Then the
p.r.p.~for $X$  is not preserved by the enlargement of filtration
from $\mathbb{F}$ to $\mathbb{G}$.
 \end{proposition}
 \begin{proof}
Let $\nu$ be a probability measure on $\mathcal{K}_T$ equivalent
to $R|_{\mathcal{K}_T}$, such that the derivative
$\frac{d\nu}{dR|_{\mathcal{K}_T}}$ is not
$\mathcal{K}_0$-measurable and
$\nu|_{\mathcal{K}_0}=R|_{\mathcal{K}_0}$.~Define $Q^{\nu}$
 on $\mathcal{G}_T$ by
$$dQ^{\nu}:=\frac{dP^X}{dR|_{\mathcal{F}_T}}\frac{d\nu}{dR|_{\mathcal{K}_T}}\,dR.$$
$Q^{\nu}|_{\mathcal{K}_T}$ is equivalent to $P|_{\mathcal{K}_T}$,
$Q^{\nu}|_{\mathcal{F}_T}$ coincides with $P^X$ and $Q^{\nu}$
decouples $\mathbb{F}$ and $\mathbb{K}$.~The conclusion follows by
Theorem 13.9 in \cite{he-wang-yan92}.
\end{proof}
\noindent The different role of condition D) under initial and
progressive enlargement  generates only an apparent paradox.~In
fact, let assume \textbf{H1)}, that is $\mathcal{F}_0$ is trivial
and the $(P^X,\mathbb{F})$-p.r.p.~for $X$ holds.~Then all source
of randomness by representing $\mathbb{F}$-local martingales lies
in $X$.~When one adds new information, either initially as in [1]
or progressively as in  Section 4 of this paper, this fact
modifies two fundamental aspects of the interpretation of $X$ in
the representation of the $\mathbb{G}$-local martingales.~First,
the randomness of $X$ is no more sufficient.~More precisely, by
initial enlargement the starting value of the martingales in the
enlarged filtration becomes random, whereas by progressive
enlargement new stochastic integrators appear.~Second, one has to
extend the probability measure $P^X$ to a measure $R$ on
$\mathcal{G}_T$, in order to get the $(R, \mathbb{G})$-p.r.p.~for
the driving process (eventually multidimensional), which in case
of progressive enlargement contains $X$ as a component and in case
of initial enlargement is $X$
itself.\vspace{0.3em}\\
\noindent Given a L\'evy process $X$, Corcuera, Nualart and
Schoutens in \cite{co-nu-sc05}
  construct a basic set of orthogonal martingales for
$\mathbb{F}^X$ (\textit{Teugels martingales})  by an
orthogonalization procedure, using $X$ and its \textit{power jump
processes} $\sum_{s\le t}\Delta X^i_s,\;i\geq 2$.~That result
suggests that in the case of a pair of L\'evy processes $(X,Y)$
the representation of $\mathbb{F}^{X}\vee\mathbb{F}^Y$-local
martingales should involve  $X$, $Y$ the processes of the form
$\sum_{s\le t}\Delta X^i_s\Delta Y^j_s,\;i,j\geq 1$.~Here, in the
more general framework of semi-martingales, Theorem
\ref{thm-SEMI-MG-COV-NO-NULLA} gives conditions for representing
all $\mathbb{F}^{X}\vee\mathbb{F}^Y$-local martingales: it is
sufficient to add to $X$ and $Y$ the process of the common jumps
$\sum_{s\le t}\Delta X_s\Delta Y_s$.~One can observe that, under
the hypotheses of Theorem \ref{thm-SEMI-MG-COV-NO-NULLA}, the
family of processes given by $X$, $Y$ and  $\sum_{s\le t}\Delta
X^i_s\Delta Y^j_s,\;i,j\geq 1$ is invariant under covariation.~In
particular when $X$ and $Y$ are martingales this family is a
\textit{compensated-stable covariation family of martingales} in
the space of square-integrable martingales  (see
\cite{ditella_enge15_teor}, \cite{ditella_enge15_stoch}).~Moreover
Theorem \ref{thm-SEMI-MG-COV-NO-NULLA}  provides the minimal
number of martingales needed for the predictable representation of
this family.
\section{Perspectives}\label{sec:perspectives}
A natural development of the current results is to investigate
whether   Theorem \ref{thm-SEMI-MG-COV-NO-NULLA} continues to hold
under weaker conditions, allowing for instance to drop
 the hypotheses \textbf{H4)} and \textbf{H5)};~this issue is the object of ongoing research.~As  in
 Remark \ref{rem-weak-condition}, some regularity of the Girsanov
 derivatives should be sufficient to state the result and to extend
 it to the multidimensional case.\bigskip\\
 In a mathematical finance environment, an interesting and natural question is to    investigate the possibility to obtain  representation results  in markets driven by
processes sharing  accessible jumps times  with positive
probability. We conjecture that this goal could be achieved
exploiting the fact that, under a decoupling measure for the
assets, the covariation process is an element of the base of
orthogonal martingales in the sense of Davis and Varaiya (\cite
{davis1}). The validation of this conjecture is another topic of
ongoing and future research.
\section{Acknowledgement} We thank Giovanna Nappo for helpful
comments.~We are in debt to her for suggesting us a way to
construct the examples. We also thank the anonymous referees whose
suggestions have resulted in a major revision of the paper and
have thus improved the presentation considerably.
\bibliography{gSSRguide}
\end{document}